\documentclass[11pt]{article}

\usepackage{mathrsfs}

\usepackage{amsmath,amsthm,amssymb}
\usepackage{graphicx}
\usepackage{vector}

\font\tencmmib=cmmib10 \skewchar\tencmmib '60
\newfam\cmmibfam
\textfont\cmmibfam=\tencmmib

\def\lessim{\ \lower4pt\hbox{$
		\buildrel{\displaystyle <}\over\sim$}\ }
\def\gessim{\ \lower4pt\hbox{$\buildrel{\displaystyle >}
		\over\sim$}\ }

\newcommand{\e}{\mathbb{E}}
\newcommand{\p}{\mathbb{P}}

\newtheorem{lemma}{\bf Lemma}

\newtheorem{theorem}{\bf Theorem}

\newtheorem{proposition}{\bf Proposition}

\newenvironment{Proof of lemma}{\noindent{\bf Proof of Lemma}}{\hfill$\Box$\newline}
\newenvironment{Proof of theorem}{\noindent{\bf Proof of Theorem}}{\hfill{\footnotesize${\square}$}\newline}
\newenvironment{Proof of theorems}{\noindent{\bf Proof of Theorems}}{\hfill$\Box$\newline}
\newenvironment{Proof of proposition}{\noindent{\bf Proof of Proposition}}{\hfill$\Box$\newline}
\newenvironment{Proof of propositions}{\noindent{\bf Proof of Propositions}}{\hfill$\Box$\newline}
\newenvironment{Proof of exercise}{\noindent{\it Proof of Exercise:}}{\hfill$\Box$}

\begin{document}
	\title{Limit Theorems in the Imitative Monomer-Dimer Mean-Field Model via Stein's Method}
	
	\author{
		Wei-Kuo Chen\thanks{Email:wkchen@math.uchicago.edu} \\
		\small{University of Chicago and University of Minnesota}
	}
	\date{}
	\maketitle
	\begin{abstract}
		We consider the imitative monomer-dimer model on the complete graph introduced in \cite{ACM14}. It was understood that this model is described by the monomer density and has a phase transition along certain critical line. By reverting the model to a weighted Curie-Weiss model with hard core interaction, we establish the complete description of the fluctuation properties of the monomer density on the full parameter space via Stein's method of exchangeable pairs. We show that this quantity exhibits the central limit theorem away from the critical line and enjoys a non-normal limit theorem  at criticality with normalized exponent $3/4$. Furthermore, our approach also allows to obtain the conditional central limit theorems along the critical line. In all these results, the Berry-Esseen inequalities for the Kolomogorov-Smirnov distance are given.
	\end{abstract}
	
	

	\section{Introduction and main results}
	
	In \cite{ACM14,ACFM15}, the authors introduced a mean-field system of interacting monomers and dimers with imitative interaction, called the imitative monomer-dimer (IMD) model. Depending on the attractive potential $J>0$ and the monomer potential $h$, this model was described by the infinite volume limit of the monomer density and was shown to exhibit a phase transition along a critical line $\Gamma.$ More precisely, the monomer density converges to a constant $m_0(J,h)$ for any $(J,h)\notin \Gamma$ and is concentrated at two distinct values $m_1(J,h)$ and $m_2(J,h)$ for $(J,h)\in\Gamma.$
	The aim of this investigation is to establish limit theorems for the monomer density on the full parameter space $(J,h)$. Following the methodology of the Curie-Weiss (CW) model \cite{EN78:1,EN78:2}, the previous known results in this direction were obtained recently in \cite{ACFM15}, where it was proven that this quantity satisfies the central limit theorem away from the critical line and possesses a non-normal limit behavior at the criticality with normalized exponent $3/4$. In the present paper, we study the limit  theorems for the monomer density in a completely different approach via Stein's method for exchangeable pairs. 
	We recover the results in \cite{ACFM15} and provide the Berry-Esseen type inequalities for the Kolmogorov-Smirnov distance. Furthermore, our approach also extends to the conditional central limit theorems along the critical line $\Gamma$ given the monomer density being above or below any fixed level $\xi$ between $m_1(J,h)$ and $m_2(J,h).$ These results together conclude a complete description of the fluctuation properties of the IMD model. 
	
    We now introduce the IMD model and state our main results as follows. For $N\geq 1$, let $C=(V  ,E  )$ be a complete graph with vertex set $V  =\{1,\ldots,N\}$ and edge set $E  =\{uv\equiv\{u,v\}:u,v\in V  ,\,u<v\}.$ A dimer configuration $D$ on $C$ is a set of edges such that $uw\notin D$ for all $w\neq v$ if $uv\in D$ and the set of monomers $\mathscr{M}(D)$, associated to $D$, is the collection of dimer-free vertices. Denote by $\mathscr{D}$ the set of all dimer configurations. Apparently, by definition, the dimer configuration and the monomer set satisfy the equation of hard core interaction,
	\begin{align}
	\label{hc}
	2|D|+|\mathscr{M}(D)|=N.
	\end{align} 
	The Hamiltonian of the IMD model with imitation coefficient $J\geq 0$ and external field $h\in\mathbb{R}$ is defined as
	\begin{align*}
	-H (D)&=N\bigl(a  m  (D)^2+b  m  (D)\bigr)
	\end{align*}
	for all $D\in \mathscr{D}  ,$ where $$
	m  (D)=\frac{|\mathscr{M}(D)|}{N}$$ is called the monomer density and the parameters $a  $ and $b  $ are given by
	\begin{align*}
	a&=J\,\,\mbox{and}\,\,b=\frac{\log N}{2}+h-J.
	\end{align*} The associated Gibbs measure and free energy are defined respectively as
	\begin{align*}
	\p(D)&=\frac{e^{-H(D)}}{\sum_{D\in \mathscr{D}  }e^{-H(D)} }
	\end{align*}
	and
	\begin{align*}
	p_N=\frac{1}{N}\log \sum_{D\in \mathscr{D}  }e^{-H(D)}.
	\end{align*}
	It is well-known that the infinite volume limit of the free energy of the IMD model is given by
	\begin{align}\label{fg}
	\lim_{N\rightarrow\infty}p_N  =\sup_{m\in[0,1]}\tilde{p}(m),
	\end{align}
	where letting 
	\begin{align}
	\begin{split}\label{g}
	g(x)&=\frac{1}{2}\bigl(\sqrt{e^{4x}+2e^{2x}}-e^{2x}\bigr)
	\end{split}\end{align}
	and
	\begin{align}\label{tau}
	\tau(x)&=(2x-1)J+h,
	\end{align}
	the function $\tilde{p}$ is defined as
	\begin{align*}
	\tilde{p}(m)&=-Jm^2-\frac{1}{2}\Bigl(1-g\circ\tau(m)+\log(1-g\circ\tau(m))\Bigr).
	\end{align*}
	In \cite{ACM14}, it has been investigated that the IMD model exhibits three different phases. We summarize the first two as follows. Let
	\begin{align*}
	J_c&=\frac{1}{4(3-2\sqrt{2})}\,\,\mbox{and}\,\,	h_c=\frac{1}{2}\log(2\sqrt{2}-2)-\frac{1}{4}.
	\end{align*}
	There exists a function $\gamma:(J_c,\infty)\rightarrow \mathbb{R}$ with $\gamma(J_c)=h_c$ such that for $\Gamma:=\{(J,\gamma(J)):J>J_c\}$, if $(J,h)\notin\Gamma$, then \eqref{fg} has a unique maximizer $m_0$ and this quantity satisfies
	\begin{align}
	\label{thm1:eq}
	m_0&=g\circ \tau(m_0).
	\end{align}
	Furthermore, if $(J,h)\neq (J_c,h_c)$, then $\tilde{p}''(m_0)<0$ and if $(J,h)=(J_c,h_c)$, then $m_0=m_c:=2-\sqrt{2}$ and
	\begin{align*}
	\tilde{p}'(m_c)=0,\,\,\tilde{p}''(m_c)=0,\,\,\tilde{p}^{(3)}(m_c)=0,\,\,\tilde{p}^{(4)}(m_c)<0.
	\end{align*}
	The importance of the maximizer lies on the fact that the monomer density satisfies the law of large numbers that $m  (D)\rightarrow m_0$ for any $(J,h)\notin \Gamma$, which can be seen either from \cite[Theorem 1.5]{ACFM15} or from Lemma \ref{prop3} below. 
	It is therefore natural to investigate the fluctuation of the monomer density, for which results related to this direction have been implemented in a recent paper \cite{ACFM15}, where the authors proved the limit theorems for any pair $(J,h)\notin \Gamma$ by adapting the classical treatment for the CW model from \cite{EN78:1,EN78:2}. Our main results here establish the same limit theorems and more importantly, give the Berry-Esseen type inequalities.
	
	\begin{theorem}\label{mainthm}
		If $(J,h)\notin \Gamma\cup\{(J_c,h_c)\}$, then there exists some constant $K$ such that
		\begin{align}\label{mainthm:eq1}
		\sup_z\Bigl|\p\Bigl(\frac{|\mathscr{M}(D)|-Nm_0}{N^{1/2}}\leq z\Bigr)-\p(X\leq z)\Bigr|\leq \frac{K}{N^{1/2}},
		\end{align} 
		where $X$ is a normal random variable with mean zero and variance $\lambda:=-\tilde{p}''(m_0)^{-1}-(2J)^{-1}>0.$
		If $(J,h)=(J_c,h_c)$, then there exists some constant $K$ such that
		\begin{align}\label{mainthm:eq2}
		\sup_z\Bigl|\p\Bigl(\frac{|\mathscr{M}(D)|-Nm_0}{N^{3/4}}\leq z\Bigr)-\p(Y\leq z)\Bigr|\leq \frac{K}{N^{1/4}},
		\end{align}
		where letting 
		$
		\lambda_c:=-\tilde{p}^{(4)}(m_c)>0,
		$
		the random variable $Y$ has density $
		ce^{-\lambda_cz^4/24}
		$
		with $c$ a normalizing constant.
	\end{theorem}
	
	The third phase is along the critical line $(J,h)\in \Gamma$. The variational formula $\eqref{fg}$ now has two distinct maximizers $m_1$ and $m_2$ with $m_1<m_2.$ They both satisfy \eqref{thm1:eq} and $\tilde{p}''(m_1),\tilde{p}''(m_2)<0,$ see \cite{ACM14}. Moreover, it was later understood in \cite{ACFM15} that the monomer density converges weakly to an atomic measure $p_1\delta_{m_1}+p_2\delta_{m_2}$ for some $p_1,p_2>0.$ Let $\xi\in (m_1,m_2).$
	Set
	\begin{align}
	\begin{split}
	\label{setd}
	\mathscr{D}_1&=\{D\in\mathscr{D}:m(D)<\xi\},\\
	\mathscr{D}_2&=\{D\in \mathscr{D}:m(D)>\xi\}.
		\end{split}
	\end{align}
	Our next result presents central limit theorems for the monomer density conditioning on $\mathscr{D}_1$ and $\mathscr{D}_2.$
	
	\begin{theorem}
		\label{thm4}
		If $(J,h)\in \Gamma,$ then there exists some $K>0$ such that 
		\begin{align*}
		\sup_z\Bigl|\p \Bigl(\frac{|\mathscr{M}(D)|-Nm_\ell}{N^{1/2}}\leq z\Big|\mathscr{D}_\ell\Bigr)-\p(X_\ell\leq z)\Bigr|\leq \frac{K}{N^{1/2}},
		\end{align*}
		where $X_\ell$ is a normal random variable with mean zero and variance $\lambda_\ell:=-\tilde{p}''(m_\ell)^{-1}-(2J)^{-1}>0.$ 
	\end{theorem}
	
    Theorems \ref{mainthm} and \ref{thm4} together describe the fluctuation properties of the model on the full parameter space. They will be established based on the general framework of the Stein method of exchangeable pairs in \cite{CS11}. This approach has been greatly used in obtaining the limit theorems for the magnetization in the classical CW model \cite{CS11,CFS13,EL09} or the mean-field Heisenberg model \cite{KM13}. The idea of our argument is to reformulate the IMD model as a CW model with additional weights described by the hard core interaction \eqref{hc}. For $N\geq 1,$ set $\Sigma =\{0,1\}^N$. Define a Hamiltonian
	\begin{align*}
	-H(\sigma)&=N(am  (\sigma)^2+bm  (\sigma))
	\end{align*}
	for any $\sigma=(\sigma_1,\ldots,\sigma_N  )\in\Sigma$,
	where $$m  (\sigma)=\frac{1}{N}\sum_{i=1}^N\sigma_i$$ is called the magnetization of the configuration $\sigma$. In addition, we denote by $\mathscr{A}(\sigma)$ the set of all sites $i\in V  $ with $\sigma_i=1$ and by ${D}(\sigma)$ the total number of admissible dimer configurations $D\in \mathscr{D}$ that satisfy $\mathscr{M}(D)=\mathscr{A}(\sigma).$ Using these notations, we introduce the Gibbs measure,
	\begin{align}\label{gibbs}
	\p(\sigma)&=\frac{D(\sigma)\exp(-H (\sigma))}{\sum_{\tau}D(\tau)\exp(-H(\tau))}.
	\end{align}
	In other words, this defines a weighted CW model on $\Sigma$ and more importantly, from the following identity
	\begin{align*}
	&\sum_{\sigma}1(|\mathscr{A}(\sigma)|=t)D(\sigma)\exp(-H  (\sigma))=\sum_{D}1(|\mathscr{M}(D)|=t)\exp(-{H} (D))
	\end{align*}
	for any $t=0,1,\ldots,N$, it satisfies  
	\begin{align}\label{eq}
	\p\Bigl(m  (\sigma)=\frac{t}{N}\Bigr)&=\p\Bigl(m  (D)=\frac{t}{N}\Bigr).
	\end{align}
	From this equation, to prove the limit theorems for the monomer density in the IMD model, it suffices to investigate the magnetization in the weighted CW model. We remark that while the space of all admissible dimer configurations $\mathscr{D}$ is not a product space, the hypercube $\Sigma$ has a nice product structure. The main difficulty one needs to handle here is the effects brought by the weights $D(\sigma)$. Indeed, following the equation of the hard core interaction \eqref{hc}, they are equal to zero if $N-|\mathscr{A}(\sigma)|$ is not even and they are given by a large combinatorial number (see \eqref{ce}) if $N-|\mathscr{A}(\sigma)|$ is even. Recall the scheme of Stein's method to establishing the limit theorems for the magnetization in the classical CW model \cite{CS11,CFS13,EL09}. One constructs the exchangeable pair for the sampled configuration $\sigma$ by choosing a site $i$ uniformly at random from $V$ and then replacing $\sigma_i$ by $\sigma_i'$, whose law follows the conditional distribution of $\sigma$ given $(\sigma_j)_{j\neq i}$ and is independent of $\sigma_i$. The key step is the conditional moment computations for the exchangeable pair, where the resulting equations are clearly related to the derivatives of the infinite volume limit of the free energy. In the present case, due to weights $D(\sigma)$, we shall construct the exchangeable pair for the sampled configuration $\sigma$ from the Gibbs measure \eqref{gibbs} by updating a pair of spins $(\sigma_i,\sigma_j)$ at a time rather than just a single spin. As one shall see, the relations between the corresponding conditional moment computations of the exchangeable pair and the derivatives of the infinite volume limit of the free energy now become very indirect. Several derivations are non-typical and more subtle computations are needed compared to those in the CW model.


	\smallskip
	\smallskip
	
	\noindent{\bf Acknowledgements.} The author thanks Pierluigi Contucci for several enlightening discussions on the monomer-dimer model and bringing the results in \cite{ACFM15} to his attention, which lead to the current work. The author is indebted to Qi-Man Shao for the fruitful discussions about the Stein's method and to the Department of Statistics in the Chinese University of Hong Kong for the hospitality during his visit. This research is supported by NSF grant DMS-1513605 and NSF-Simon Travel Grant 2014-2015.

	\section{Stein's method}	
	
	  In this section, we describe the formulation of Stein's method from \cite{CS11}. The exchangeable pairs we shall use in this paper will be constructed through the following general proposition.
	  
	  \begin{proposition}\label{prop0}
	  	Let $\zeta=(\zeta_1,\ldots,\zeta_N)$ be a $N$-dimensional random vector and $uv$ be sampled uniformly at random from $E.$ Let $(\zeta_u',\zeta_v')$ be the conditional distribution of $(\zeta_u,\zeta_v)$ given $(\zeta_i)_{i\neq u,v}$ and be independent of $(\zeta_u,\zeta_v).$ Set $\zeta'=(\zeta_1',\ldots,\zeta_N')$ with $\zeta_i'=\zeta_i$ for all $i\neq u,v$. Then $\zeta$ and $\zeta'$ are exchangeable. 
	  \end{proposition}
	  
	  \begin{proof}
	  	Let $F,G$ be any bounded measurable functions and denote by $\mathscr{F}_{uv}$ the $\sigma$-algebra generated by $(\zeta_i)_{i\neq u,v}.$ Note that conditioning on $\mathscr{F}_{uv}$, $(\zeta_u,\zeta_v)$ and $(\zeta_v',\zeta_v')$ are i.i.d.  Consequently, the exchangeablility of $\zeta$ and $\zeta'$ follows by
	  	\begin{align*}
	  	E[F(\zeta)G(\zeta')]&=\frac{1}{|E|}\sum_{uv\in E}E\bigl[E[F(\zeta)G(\zeta')|\mathscr{F}_{uv}]\bigr]\\
	  	&=\frac{1}{|E|}\sum_{uv\in E}E\bigl[E[F(\zeta)|\mathscr{F}_{uv}]E[G(\zeta')|\mathscr{F}_{uv}]\bigr]\\
	  	&=\frac{1}{|E|}\sum_{uv\in E}E\bigl[E[F(\zeta')|\mathscr{F}_{uv}]E[G(\zeta)|\mathscr{F}_{uv}]\bigr]\\
	  	&=E[F(\zeta')G(\zeta)].
	  	\end{align*}
	  \end{proof}
	  
	Suppose that 
	$g(t)$ is nondecreasing and $g(t)\geq 0$ for $t>0$ and $g(t)\leq 0$ for $t\leq 0.$
	Let $Z$ be a random variable with density
	\begin{align*}
	p(t)&=c_1e^{-c_0\int_0^tg(s)ds}
	\end{align*}
	for $t\in\mathbb{R}$, where $c_0>0$ and $c_1$ is the normalizing constant. Let $\Delta=W-W'.$ Then Stein's method for exchange pair yields the following Berry-Esseen type inequality.
	
	\begin{theorem}[Theorem 1.2 \cite{CS11}]
		\label{thm1} Let $(W,W')$ be an exchangeable pair. Assume that there exist two real-valued functions $g$ and $r$ on $\mathbb{R}$ such that
		\begin{align}\label{thm1:ass}
		E[W-W'|W]&=g(W)+r(W),
		\end{align}
		Suppose that there exists $c_2<\infty$ such that 
		\begin{align}\label{thm1:ass1}
		c_0|g'(x)|\Bigl(|x|+\frac{3}{c_1}\Bigr)\min\Bigl(\frac{1}{c_1},\frac{1}{|c_0g(x)|}\Bigr)\leq c_2,\,\,\forall x.
		\end{align}
		If $|\Delta|\leq \delta,$ then 
		\begin{align}
		\begin{split}\label{thm1:ineq}
		\sup_z |P(W\leq z)-P(Z\leq z)|
		&\leq 3E\Bigl|1-\frac{c_0}{2}E[\Delta^2|W]\Bigr|+\frac{2c_0}{c_1}E|r(W)|\\
		&+c_1\max(1,c_2)\delta+\delta^3c_0\Bigl\{\Bigl(2+\frac{c_2}{2}E|c_0g(W)|\Bigr)+\frac{c_1c_2}{2}\Bigr\}.
		\end{split}
		\end{align}
	\end{theorem}

	\section{Moment computations away from the critical line $\Gamma$}   
    Throughout the rest of the paper, we shall use $K$ to stand for a positive constant that is independent of $N$ and could be different at each occurrence. 
	For any $\tau\in\Sigma$, $uv\in E $ and $s,t=0,1,$ we use the notation $\tau_{uv}^{st}$ to denote the configuration $\rho\in\Sigma$ that satisfies $\rho_i=\tau_i$ for all $i\neq u,v$ and $\rho_u=s$ and $\rho_v=t$. Let us sample $\sigma$ from $\p $ and let $uv$ be sampled uniformly at random from $E$. We define $(\sigma_u',\sigma_v')$ as the conditional distribution of $(\sigma_u,\sigma_v)$ given $(\sigma_i)_{i\neq u,v}$ and independent of $(\sigma_u,\sigma_v).$ In other words,
	\begin{align*}
	\p(\sigma_u'=s,\sigma_v'=t|\sigma)&=\frac{\p(\sigma_{uv}^{st})}{\p(\sigma_{uv}^{11})+\p(\sigma_{uv}^{10})+\p(\sigma_{uv}^{01})+\p(\sigma_{uv}^{00})}.
	\end{align*}
	Note that any dimer configuration $D$ with $\mathscr{M}(D)=\mathscr{A}(\sigma)$ satisfies the equation of the hard core interaction,
	$$
	2|D|+\sum_{i=1}^N\sigma_i=N,
	$$
	which deduces that 
	\begin{align*}
	D(\sigma_{uv}^{10})&=D(\sigma_{uv}^{01})=0,\,\,\mbox{if $\sigma_u=\sigma_v=1$ or $\sigma_u=\sigma_v=0$},\\
	D(\sigma_{uv}^{11})&=D(\sigma_{uv}^{00})=0,\,\,\mbox{if $\sigma_u=1,\sigma_v=0$ or $\sigma_u=0,\sigma_v=1.$}
	\end{align*} 
	Consequently, if $\sigma_u=\sigma_v=1,$ 
	\begin{align}
	\begin{split}
	\label{A}
	\p  (\sigma_u'=\sigma_v'=1|\sigma)&=\frac{\p  (\sigma_{uv}^{11})}{\p  (\sigma_{uv}^{11})+\p  (\sigma_{uv}^{00})}
	=\frac{D(\sigma)}{D(\sigma)+D(\sigma_{uv}^{00})e^{-4a  m  (\sigma)+4a  /N-2b  }},\\
	\p  (\sigma_u'=\sigma_v'=0|\sigma)&=\frac{\p  (\sigma_{uv}^{00})}{\p  (\sigma_{uv}^{11})+\p (\sigma_{uv}^{00})}
	=\frac{D(\sigma_{uv}^{00})e^{-4a  m  (\sigma)+4a  /N-2b  }}{D(\sigma)+D(\sigma_{uv}^{00})e^{-4a  m  (\sigma)+4a  /N-2b  }},
		\end{split}
	\end{align}
	if $\sigma_u=\sigma_v=0$,
	\begin{align}
	\begin{split}
	\label{B}
	\p (\sigma_u'=1,\sigma_v'=1|\sigma)&=\frac{\p  (\sigma_{uv}^{11})}{\p  (\sigma_{uv}^{11})+\p  (\sigma_{uv}^{00})}
	=\frac{D(\sigma_{uv}^{11})e^{4am(\sigma)+4a/N+2b}}{D(\sigma)+D(\sigma_{uv}^{11})e^{4am(\sigma)+4a/N+2b}},\\
	\p  (\sigma_u'=0,\sigma_v'=0|\sigma)&=\frac{\p (\sigma_{uv}^{00})}{\p(\sigma_{uv}^{11})+\p (\sigma_{uv}^{00})}
	=\frac{D(\sigma)}{D(\sigma)+D(\sigma_{uv}^{11})e^{4am(\sigma)+4a/N+2b}},
		\end{split}
	\end{align}
	and if $\sigma_u=1,\sigma_v=0$,
	\begin{align}
	\begin{split}
	\label{C}
	\p (\sigma_u'=1,\sigma_v'=0|\sigma)&=\frac{\p  (\sigma_{uv}^{10})}{\p  (\sigma_{uv}^{10})+\p(\sigma_{uv}^{01})}=\frac{D(\sigma_{uv}^{10})}{D(\sigma_{uv}^{10})+D(\sigma_{uv}^{01})}=\frac{1}{2},\\
	\p  (\sigma_u'=0,\sigma_v'=1|\sigma)&=\frac{\p  (\sigma_{uv}^{01})}{\p  (\sigma_{uv}^{10})+\p (\sigma_{uv}^{01})}=\frac{D(\sigma_{uv}^{01})}{D(\sigma_{uv}^{10})+D(\sigma_{uv}^{01})}=\frac{1}{2}.
		\end{split}
	\end{align}
	To sum up, if $u,v$ are either both monomers or both vertices of some dimers, then they only could be updated as either both monomers or both vertices of some dimers. If only one of $u,v$ is a monomer, then after the update they again has only one monomer.
	Let $\sigma'$ be the random vector obtained by replacing $\sigma_u$ and $\sigma_v$ by $\sigma_u'$ and $\sigma_v'$, respectively. From Proposition \ref{prop0}, $\sigma$ and $\sigma'$ are exchangeable. The proposition below will play an essential role to control the Berry-Esseen bounds \eqref{thm1:ineq}.

	\begin{proposition}\label{lem1} Let $M=\sum_{i=1}^N\sigma_i$ and $M'=\sum_{i=1}^N\sigma_i'.$ We have that 
		\begin{align}
		\begin{split}\label{lem1:eq1}
		\e[M-M'|\sigma]&=L_1(m(\sigma))+R_1(m(\sigma)),
		\end{split}\\
		\begin{split} \label{lem2:eq1}
		\e[(M-M')^2|\sigma]&=L_2(m(\sigma))+R_2(m(\sigma)),
		\end{split}
		\end{align}
		where recalling $\tau(m)$ from \eqref{tau}, $L_1,L_2,R_1,R_2$ satisfy that for all $m\in[0,1],$
		\begin{align*}
		&L_1(m)=\frac{2(1-m)(m^2-(1-m)e^{2\tau(m)})}{(1-m)+e^{2\tau(m)}},\\
		&L_2(m)=\frac{4(1-m)(m^2+(1-m)e^{2\tau(m)})}{(1-m)+e^{2\tau(m)}},\\
		&|R_1(m)|,\,\,|R_2(m)|\leq \frac{K}{N}
		\end{align*}
		for some constant $K>0.$ 
	\end{proposition}

	\begin{proof} Let $k=1,2.$ Consider
		\begin{align*}
		\e[(M-M')^k|\sigma]&=\frac{1}{|E|}\sum_{uv\in E}\e[(\sigma_u+\sigma_v-\sigma_u'-\sigma_v')^k|\sigma]\\
		&=\frac{1}{|E|}\sum_{u,v\in \mathscr{A}(\sigma):u<v}\e[(\sigma_u+\sigma_v-\sigma_u'-\sigma_v')^k|\sigma]\\
		&+\frac{1}{|E|}\sum_{u\in \mathscr{A}(\sigma),v\notin \mathscr{A}(\sigma)}\e[(\sigma_u+\sigma_v-\sigma_u'-\sigma_v')^k|\sigma]\\
		&+\frac{2}{|E|}\sum_{u,v\notin \mathscr{A}(\sigma):u<v}\e[(\sigma_u+\sigma_v-\sigma_u'-\sigma_v')^k|\sigma].
		\end{align*}
		We compute each summation as follows.
		For $u\in \mathscr{A}(\sigma)$ and $v\notin \mathscr{A}(\sigma),$ since $(\sigma_u',\sigma_v')$ could only be either $(1,0)$ or $(0,1)$ from \eqref{C}, it follows that
		\begin{align*}
		\e[(\sigma_u+\sigma_v-\sigma_u'-\sigma_v')^k|\sigma]&=\e[(1-1)^k|\sigma]=0.
		\end{align*}
		To compute the first and third summations, note that the total number of $L/2$ dimers on a complete graph of size $L$ can be computed as
		\begin{align}\label{ce}
		\frac{1}{(L/2)!}{L\choose 2}{L-2\choose 2}\cdots {L-2(L/2-1)\choose 2}=\frac{L!}{(L/2)!}2^{-L/2}.
		\end{align}
		Take $L=N-|\mathscr{A}(\sigma)|.$
		For $u,v\in \mathscr{A}(\sigma)$ with $u<v$, since $(\sigma_u',\sigma_v')$ could only be either $(1,1)$ or $(0,0)$ from \eqref{A}, this yields
		\begin{align*}
		\e[(\sigma_u+\sigma_v-\sigma_u'-\sigma_v')^k|\sigma]&=2^k\p(\sigma_u'=0,\sigma_v'=0|\sigma)\\
		&=\frac{2^kD(\sigma_{uv}^{00})e^{-4am(\sigma)+4a/N-2b}}{D(\sigma)+D(\sigma_{uv}^{00})e^{-4am(\sigma)+4a/N-2b}}\\
		&=\frac{2^k\frac{(L+2)!}{(L/2+1)!}2^{-(L/2+1)}e^{-4am(\sigma)+4a/N-2b}}{\frac{L!}{(L/2)!}2^{-L/2}+\frac{(L+2)!}{(L/2+1)!}2^{-(L/2+1)}e^{-4am(\sigma)+4a/N-2b}}\\
		&=\frac{2^k(L+1)e^{-4am(\sigma)+4a/N-2b}}{1+(L+1)e^{-4am(\sigma)+4a/N-2b}},
		\end{align*}
		while for $u,v\notin \mathscr{A}(\sigma)$ with $u<v$, since $(\sigma_u',\sigma_v')$ could only be either $(1,1)$ and $(0,0)$ from \eqref{B}, we conclude
		\begin{align*}
		\e[(\sigma_u+\sigma_v-\sigma_u'-\sigma_v')^k|\sigma]&=(-2)^k\p(\sigma_u'=1,\sigma_v'=1|\sigma)\\
		&=\frac{(-2)^kD(\sigma_{uv}^{11})e^{4am(\sigma)+4a/N+2b}}{D(\sigma)+D(\sigma_{uv}^{11})e^{4am(\sigma)+4a/N+2b}}\\
		&=\frac{(-2)^k\frac{(L-2)!}{(L/2-1)!}2^{-(L/2-1)}e^{4am(\sigma)+4a/N+2b}}{\frac{L!}{(L/2)!}2^{-L/2}+\frac{(L-2)!}{(L/2-1)!}2^{-(L/2-1)}e^{4am(\sigma)+4a/N+2b}}\\
		&=\frac{(-2)^ke^{4am(\sigma)+4a/N+2b}}{(L-1)+e^{4am(\sigma)+4a/N+2b}}.
		\end{align*}
		Combining these two equations together, we obtain
		\begin{align*}
		&\e[(M-M')^k|\sigma]\\
		&=\frac{1}{|E|}\left({|\mathscr{A}(\sigma)|\choose 2}\frac{2^k(L+1)e^{-4am(\sigma)+4a/N-2b}}{1+(L+1)e^{-4am(\sigma)+4a/N-2b}}+{|\mathscr{A}(\sigma)^c|\choose 2}\frac{(-2)^ke^{4am(\sigma)+4a/N+2b}}{(L-1)+e^{4am(\sigma)+4a/N+2b}}\right)\\
		&=\frac{2^{k-1}}{|E|}\Bigl(\frac{M(\sigma)(M(\sigma)-1)(L+1)e^{-4am(\sigma)+4a/N-2b}}{1+(L+1)e^{-4am(\sigma)+4a/N-2b}}+\frac{(-1)^kL(L-1)e^{4am(\sigma)+4a/N+2b}}{(L-1)+e^{4am(\sigma)+4a/N+2b}}\Bigr)\\
		&=\frac{2^{k-1}N^2}{|E|}\cdot\frac{m(\sigma)(m(\sigma)-1/N)(1-m(\sigma)+1/N)e^{-4am(\sigma)+4a/N-2b}}{1/N+(1-m(\sigma)+1/N)e^{-4am(\sigma)+4a/N-2b}}\\
		&\quad+\frac{(-1)^k2^{k-1}N}{|E|}\cdot\frac{(1-m(\sigma))(1-m(\sigma)-1/N)e^{4am(\sigma)+4a/N+2b}}{(1-m(\sigma)-1/N)+e^{4am(\sigma)+4a/N+2b}/N}.
		\end{align*}
		Substituting $a=J$, $e^{2b}=e^{\log N+2h-2J}=Ne^{2h-2J}$ and $|E|=N(N-1)/2$ into this equation gives
		\begin{align*}
		\e[(M-M')^k|\sigma]
		&=\frac{2^{k}}{(1-1/N)}\cdot\frac{m(\sigma)(m(\sigma)-1/N)(1-m(\sigma)+1/N)e^{-2\tau(m(\sigma))+4J/N}/N}{1/N+(1-m(\sigma)+1/N)e^{-2\tau(m(\sigma))+4J/N}/N}\\
		&\quad+\frac{(-2)^k}{N(1-1/N)}\cdot\frac{(1-m(\sigma))(1-m(\sigma)-1/N)e^{2\tau(m(\sigma))+4J/N}N}{(1-m(\sigma)-1/N)+e^{2\tau(m(\sigma))+4J/N}}\\
		&=U_k(m(\sigma),1/N),
		\end{align*}
		where for $0\leq m\leq 1$ and small $t$,
		\begin{align*}
		U_k(m,t)&:=\frac{2^k}{(1-t)}\cdot\frac{m(m-t)(1-m+t)e^{-2\tau(m)+4Jt}}{1+(1-m+t)e^{-2\tau(m)+4Jt}}\\
		&+\frac{(-2)^k}{(1-t)}\cdot\frac{(1-m)(1-m-t)e^{2\tau(m)+4Jt}}{(1-m-t)+e^{2\tau(m)+4Jt}}.
		\end{align*}
		Note that 
		\begin{align*}
		U_k(m,0)&=\frac{2^km^2(1-m)e^{-2\tau(m)}}{1+(1-m)e^{-2\tau(m)}}+\frac{(-2)^k(1-m)^2e^{2\tau(m)}}{(1-m)+e^{2\tau(m)}}\\
		&=\frac{2^k(m^2(1-m)+(-1)^k(1-m)^2e^{2\tau(m)})}{(1-m)+e^{2\tau(m)}}\\
		&=L_k(m).
		\end{align*}
		Letting $R_k(m)=\int_0^{1/N}\partial_tU_k(m,t)dt$, we have by the fundamental theorem of calculus,
		\begin{align*}
		\e[(M-M')^k|\sigma]&=L_k(m(\sigma))+R_k(m(\sigma)),
		\end{align*}
		where since the numerators in $U_k$ stays away from zero, there exists some $K$ such that $|R_k(m)|\leq K/N.$ This finishes our proof.  
	\end{proof}

	\section{Proof of Theorem \ref{mainthm}} 
    Suppose that $m_0$ is the unique maximizer of \eqref{fg}. Recall $M,M'$ from Proposition \ref{lem1}. For $k=0$ or $1$, we set 
	$$
	W=\frac{M-Nm_0}{N^{(2k+1)/(2k+2)}}\,\,\mbox{and}\,\,W'=\frac{M'-Nm_0}{N^{(2k+1)/(2k+2)}}.
	$$
	From the previous section, it is easy to see that $(W,W')$ is exchangeable. The following lemma is the central ingredient of our argument.
	
	\begin{lemma}
		\label{lem3}
			Suppose that $m_0$ is the unique maximizer of \eqref{fg} and that for some integer $k=0$ or $1$, $$
			L_1^{(\ell)}(m_0)=0
			$$
			for all $0\leq \ell\leq 2k$
			and
			\begin{align*}
			L_1^{(2k+1)}(m_0)> 0.
			\end{align*}
		We have that
		\begin{align}\label{lem3:eq1}
		\e[W-W'|W]&=g(W)+r(W),
		\end{align}
		where \begin{align}\label{eq7}
		g(W)&=\frac{L_1^{(2k+1)}(m_0)}{(2k+1)!N^{(2k+1)/(k+1)}}W^{2k+1}
		\end{align}
		and $r$ is the remainder term satisfying 
		\begin{align}\label{eq3}
		|r(W)|\leq\frac{K}{N^{(4k+3)/(2k+2)}}\Bigl(W^{2k+2}+1\Bigr).
		\end{align}
		In addition,
		\begin{align}
		\label{lem3:eq2}
		\e\Bigl|1-\frac{c_0}{2}\e[(W-W')^2|W]\Bigr|&\leq K\Bigl(\frac{1}{N^{1/(2k+2)}}+\frac{1}{N}\Bigr),
		\end{align}
		where \begin{align}
		\label{c_0}
		c_0&=\frac{2N^{(2k+1)/(k+1)}}{L_2(m_0)}.
		\end{align}
		
	\end{lemma}
	
		It should be pointed out that in the CW model \cite{CS11,EL09}, the exchangeable pairs were constructed by choosing a single site $i$ uniformly at random and updating $\sigma_i$ by $\sigma_i'$, whose law follows the conditional distribution of $\sigma_i$ given $(\sigma_j)_{j\neq i}.$ In those cases, the function $g$ can be expressed in terms of the second or fourth derivative of the infinite volume limit of the free energy, but this is not the case in the IMD model.

	\begin{proof}[Proof of Lemma \ref{lem3}]
		From the given assumptions, the Taylor formula yields
		\begin{align*}
		L_1(m(\sigma))&=\frac{L_1^{(2k+1)}(m_0)}{(2k+1)!}(m(\sigma)-m_0)^{2k+1}+\frac{\int_{m_0}^{m(\sigma)}L_1^{(2k+2)}(s)(m(\sigma)-s)^{2k+1}ds}{(2k+1)!}.
		\end{align*}
		Since
		\begin{align*}
		m(\sigma)-m_0&=\frac{W}{N^{1/(2k+2)}}\,\,\mbox{and}\,\,m(\sigma')-m_0=\frac{W'}{N^{1/(2k+2)}},
		\end{align*}
		we have that from \eqref{lem1:eq1},
		\begin{align*}
		\e[W-W'|W]&=\frac{L_1(m(\sigma))}{N^{(2k+1)/(2k+2)}}+\frac{R_1(m(\sigma))}{N^{(2k+1)/(2k+2)}}=g(W)+r(W),
		\end{align*}
		where $g$ is given by \eqref{eq7} and 	\begin{align*}
		r(W)&=\frac{\int_{m_0}^{m(\sigma)}L_1^{(2k+2)}(s)(m(\sigma)-s)^{2k+1}ds}{(2k+1)!N^{(2k+1)/(2k+2)}}+\frac{1}{N^{(2k+1)/(2k+2)}}R_1\Bigl(\frac{W}{N^{1/(2k+2)}}+m_0\Bigr).
		\end{align*}
		Here \eqref{eq3} follows by
		\begin{align*}
		|r(W)|&\leq K\Bigl(\frac{(m(\sigma)-m_0)^{2k+2}}{(2k+1)!N^{(2k+1)/(2k+2)}}+\frac{1}{N^{(2k+1)/(2k+2)}}\cdot\frac{1}{N}\Bigr)
		=\frac{K}{N^{(4k+3)/(2k+2)}}\Bigl(W^{2k+2}+1\Bigr).
		\end{align*}
		To show \eqref{lem3:eq2}, we use \eqref{lem2:eq1} and the fundamental theorem of calculus to obtain	
		\begin{align*}
		\e[(W-W')^2|W]&=\frac{L_2(m(\sigma))}{N^{(2k+1)/(k+1)}}+\frac{R_2(m(\sigma))}{N^{(2k+1)/(k+1)}}\\
	&=\frac{2}{c_0}+\int_{m_0}^{m(\sigma)}\frac{L_2'(s)}{N^{(2k+1)/(k+1)}}ds+\frac{R_2(m(\sigma))}{N^{(2k+1)/(k+1)}}.
		\end{align*}
		and therefore, 
		\begin{align*}
		\e \Bigl|1-\frac{c_0}{2}\e[(W-W')^2|W]\Bigr|&=\frac{c_0}{2}\e\Bigl|\int_{m_0}^{m(\sigma)}\frac{L_2'(s)}{N^{(2k+1)/(k+1)}}ds+\frac{R_2(m(\sigma))}{N^{(2k+1)/(k+1)}}\Bigr|\\
		&\leq K\Bigl(\e|m(\sigma)-m_0|+\frac{1}{N}\Bigr)\\
		&\leq K\Bigl(\frac{\e |W|}{N^{1/(2k+2)}}+\frac{1}{N}\Bigr).
		\end{align*}
	\end{proof}
	
	Next we need an auxiliary lemma. Denote by $O$ the collection of all maximizers to \eqref{fg}. Note that from the introduction $O$ contains at most two points. 
	For $m\in[0,1]$, let $d(m)$ be the distance from $m$ to the set $O.$
	
	\begin{lemma}\label{prop3}
        For any $\delta>0$, there exists $\eta>0$ such that
		$$
		\p(d(m(\sigma))\geq \delta)\leq Ke^{-N\eta}.
		$$
	\end{lemma}
	
	\begin{proof} Let $U=\{m\in[0,1]:d(m)\geq \delta\}.$ By the virtue of \eqref{eq}, it suffices to prove that for any $\delta>0$, there exists $\eta>0$ such that
		$$
		\p(m(D)\in U)\leq Ke^{-N\eta}.
		$$
		Note that
		\begin{align*}
		\frac{1}{N}\log\p(m(D)\in U)&\leq \frac{1}{N}\log\sum_{D:m(D)\in U}\exp(-H(D))-\frac{1}{N}\log\sum_{D}\exp(-H(D)).
		\end{align*}
		Set $A=\{0,1/N,\ldots,(N-1)/N,1\}.$ Observe that
		\begin{align*}
		\delta_{m(D),m}\exp(-H(D))&=\delta_{m(D),m}\exp N(am(D)^2+bm(D))\\
		&=\delta_{m(D),m}\exp N(a(2m(D)m-m^2)+bm(D))\\
		&=\delta_{m(D),m}\exp N((2am+b)m(D)-am^2).
		\end{align*}
		We obtain
		\begin{align*}
		\sum_{D:m(D)\in U}\exp(-H(D))&=\sum_{D}1(m(D)\in U)\exp(-H(D))\\
		&=\sum_{m\in A  \cap U}\sum_{D}\delta_{m(D),m}\exp(-H(D))\\
		&=\sum_{m\in A  \cap U}\sum_{D}\exp N((2am+b)m(D)-am^2)\\
		&\leq (N+1)\sup_{m\in A\cap U}e^{-aNm^2}\sum_{D}\exp N(2am+b)m(D)
		\end{align*}
		and thus,
		\begin{align*}
		&\frac{1}{N}\log \p((m(\sigma)\in U)\\
		&\leq \frac{\log (N+1)}{N}+\sup_{m\in U}\Bigl\{-am^2+\frac{1}{N}\log\sum_{D}\exp \bigl(N(2am+b)m(D)\bigr)\Bigr\}-\frac{1}{N}\log \sum_D\exp(-H(D)).
		\end{align*}
		Here,
		$$
		\frac{1}{N}\log\sum_{D}\exp \bigl(N(2am+b)m(D)\bigr)
		$$
		is indeed the free energy of an IMD model with parameter $(J',h')=(0,(2m-1)J+h)$ and its thermodynamic limit, according to the formula \eqref{fg}, is equal to
		$$
		-\Bigl(\frac{1-g\circ\tau(m)}{2}+\log(1-g\circ\tau(m))\Bigr).
		$$
		As a consequence, $$
		\limsup_{N\rightarrow\infty}\frac{1}{N}\log \p(m(D)\in U)\leq \sup_{m\in U}\tilde{p}(m)-\sup_{m\in[0,1]}\tilde{p}(m)=:-2\eta.
		$$
		Since $\tilde{p}(m)$ is continuous on $[0,1]$ and any point in $O$ is away from the maximizers with distance at least $\delta,$ we conclude that $\eta>0$ and that there exists some $N_0$ such that for all $N\geq N_0,$
		\begin{align*}
		\frac{1}{N}\log \p\bigl(m(D)\in U\bigr)&\leq -\eta
		\end{align*}
		and consequently, 
		$\p(m(D)\in U)\leq Ke^{-N\eta}.$
	\end{proof}
	
	\begin{lemma}\label{lem5}
	   Under the assumptions of Proposition \ref{lem1}, there exists some $K>0$ such that $\e W^{2k+2}\leq K$ for all $N\geq 1.$
	\end{lemma}
	
	\begin{proof}
		From \eqref{lem3:eq1}, we have that
		\begin{align*}
		W^{2k+1}&=\frac{(2k+1)!N^{(2k+1)/(k+1)}}{L_1^{(2k+1)}(m_0)}\Bigl(\e[W-W'|W]-r(W)\Bigr).
		\end{align*}
		Multiplying $W$ on both sides and then taking expectation give
		\begin{align}
		\begin{split}
		\notag
		\e W^{2k+2}&=\frac{(2k+1)!N^{(2k+1)/(k+1)}}{L_1^{(2k+1)}(m_0)}\Bigl(\e[(W-W')W]-\e Wr(W)\Bigr)
		\end{split}\\
		\begin{split}
		\label{eq4}
		&\leq \frac{(2k+1)!}{L_1^{(2k+1)}(m_0)}\Bigl(N^{(2k+1)/(k+1)}\e[(W-W')W]+\frac{K\e |W|^{2k+3}}{N^{1/(2k+2)}}+\frac{K\e|W|}{N^{1/(2k+2)}}\Bigr),
		\end{split}
		\end{align}
		where we have used \eqref{eq3} to bound $r(W).$
		Here from the exchangeablility of $W$ and $W'$, we can express $\e[(W-W')W]=2^{-1}\e(W-W')^2$, which combined with the trivial bound $|W-W'|\leq 2/N^{(2k+1)/(2k+2)}$ allows to control the first term of \eqref{eq4},
		$$
		N^{(2k+1)/(k+1)}\e[(W-W')W]\leq 2.
		$$ 
		As for the third term, we use the bond $|W|\leq N^{1/(2k+2)}$ to obtain $N^{-1/(2k+2)}\e|W|\leq 1.$ To bound the second term, for any $\delta>0$, Lemma \ref{prop3} says that there exists some $\eta>0$ and $K>0$ such that
		\begin{align*}
		\p(|W|\geq \delta N^{1/(2k+2)})=\p(|m(\sigma)-m_0|\geq \delta)&\leq Ke^{-N\eta}
		\end{align*}
		for all $N\geq 1.$ Consequently, using again the trivial bound $|W|\leq N^{1/(2k+2)},$
		\begin{align}
		\begin{split}\label{lem3:proof:eq1}
		\frac{\e |W|^{2k+3}}{N^{1/(2k+2)}}&=\frac{\e [|W|^{2k+3};|W|\leq \delta N^{1/(2k+2)}]}{N^{1/(2k+2)}}+\frac{\e [|W|^{2k+3};|W|\geq \delta N^{1/(2k+2)}]}{N^{1/(2k+2)}}\\
		&\leq \delta \e|W|^{2k+2}+N\p(|m(\sigma)-m_0|\geq \delta)\\
		&\leq \delta \e|W|^{2k+2}+KNe^{-\eta N}.
		\end{split}
		\end{align}
		Plugging these three bounds into \eqref{eq4} gives
		\begin{align*}
		\Bigl(1-\frac{K(2k+1)!}{L_1^{(2k+1)}(m_0)}\delta\Bigr)\e|W|^{2k+2}&\leq \frac{(2k+1)!}{L_1^{(2k+1)}(m_0)}\Bigl(2+KNe^{-\eta N}+K\Bigr),
		\end{align*}
		which completes our proof.
		
	\end{proof}

	\begin{lemma}\label{lem4}
	   Suppose that the conditions of Proposition \ref{lem1} hold.
		Let $Z$ be a continuous random variable on $\mathbb{R}$ with density
		\begin{align*}
		p(z)&=c_1\exp\Bigl(-dz^{2k+2}\Bigr),
		\end{align*}
		where $$
		d:=\frac{2L_1^{(2k+1)}(m_0)}{(2k+2)!L_2(m_0)}.
		$$
		and $c_1$ is a normalizing constant. Then there exists some constant $K$ independent of $N$ such that
		\begin{align*}
		\sup_{z}\Bigl|\p(W\leq z)-\p(Z\leq z)\Bigr|\leq \frac{K}{N^{1/(2k+2)}}.
		\end{align*}
	\end{lemma}
	
	\begin{proof}
		Recall $c_0$ from \eqref{c_0} and $g,r$ from \eqref{lem3:eq1}. We define 
		\begin{align*}
		p(t)&=c_1e^{-c_0\int_0^tg(s)ds}=c_1e^{-dt^{2k+2}},
		\end{align*}
		where $c_1$ is a normalizing constant such that $p$ is a probability density on $\mathbb{R}.$ Using these $c_0,c_1,g,r$, we now verify \eqref{thm1:ass1} for some $c_2$, which can be easily seen since
		\begin{align*}
		&c_0|g'(x)|\Bigl(|x|+\frac{3}{c_1}\Bigr)\min\Bigl(\frac{1}{c_1},\frac{1}{|c_0g(x)|}\Bigr)\\
		&=(2k+2)d|x|^{2k+1}\Bigl(|x|+\frac{3}{c_1}\Bigr)\min\Bigl(\frac{1}{c_1},\frac{1}{dx^{2k+2}}\Bigr)
		\end{align*}
		has a limit at infinity and is clearly bounded for arbitrary small $x.$ As a result, the inequality \eqref{thm1:ineq} with $\delta=2/N^{(2k+1)/(2k+2)}$ gives
		\begin{align*}
		&\sup_z|\p(W\leq z)-\p(Z\leq z)|\\
		&\leq 3K\Bigl(\frac{1}{N^{1/(2k+2)}}+\frac{1}{N}\Bigr)+\frac{4K}{N^{1/(2k+2)}}(\e |W|^{2k+2}+1)\\
		&+\frac{2c_1\max(1,c_2)}{N^{(2k+1)/(2k+2)}}+\frac{16}{N^{(2k+1)/(2k+2)}L_2(m_0)}\Bigl\{\Bigl(2+\frac{c_2d}{2}\e |W|^{2k+1}\Bigr)+\frac{c_1c_2}{2}\Bigr\}\\
		&\leq \frac{K}{N^{1/(2k+2)}},
		\end{align*}
		where the first inequality used Lemma \ref{lem3} and the second one used Lemma \ref{lem5}.
		
	\end{proof}

	\begin{proof}[Proof of Theorem \ref{mainthm}] Recall $\lambda,\lambda_c$ from Theorem \ref{mainthm}. Suppose that $(J,h)\notin\Gamma\cup\{(J_c,h_c)\}$ and $m_0$ is the unique maximizer of \eqref{fg}. From \eqref{g}, $m_0$ satisfies
		\begin{align}\label{eq8}
		2m_0+e^{2\tau(m_0)}=\sqrt{e^{4\tau(m_0)}+4e^{2\tau(m_0)}}
		\end{align}
		or equivalently
		\begin{align}\label{eq9}
		m_0^2=(1-m_0)e^{2\tau(m_0)}.
		\end{align}
		Note that from \eqref{eq8},
		\begin{align*}
		\tilde{p}''(m_0)
		&=2J(2Jg'\circ\tau(m_0)-1)\\
		&=2J\Bigl(2J\Bigl(\frac{e^{4\tau(m_0)}+2e^{2\tau(m_0)}}{\sqrt{e^{4\tau(m_0)}+4e^{2\tau(m_0)}}}-e^{2\tau(m_0)}\Bigr)-1\Bigr)\\
		&=-2J\frac{2m_0+(4J(m_0-1)+1)e^{2\tau(m_0)}}{2m_0+e^{2\tau(m_0)}}
		\end{align*}  
		and thus,
		\begin{align*}
		\lambda&=\Bigl(-\frac{1}{\tilde{p}''(m_0)}-\frac{1}{2J}\Bigr)=\frac{2(1-m_0)e^{2\tau(m_0)}}{2m_0+(4J(m_0-1)+1)e^{2\tau(m_0)}}.
		\end{align*}
		On the other hand, \eqref{eq9} implies that $L_1(m_0)=0$ and that
		\begin{align*}
		\frac{2L_1'(m_0)}{2!L_2(m_0)}&=\frac{1}{2}\frac{2m_0+(4J(m_0-1)+1)e^{2\tau(m_0)}}{m_0^2+(1-m_0)e^{2\tau(m_0)}}=\frac{2m_0+(4J(m_0-1)+1)e^{2\tau(m_0)}}{4(1-m_0)e^{2\tau(m_0)}}=\frac{1}{2\lambda}.
		\end{align*}
		Since the equation \eqref{eq8} also implies
		\begin{align*}
		\tilde{p}''(m_0)+2J&=2Jg'\circ\tau(m_0)=\frac{4J(1-m_0)e^{2\tau(m_0)}}{2m_0+e^{2\tau(m_0)}}>0,
		\end{align*}
		we conclude that $\lambda>0$ and thus, $L_1'(m_0)>0.$ Lemma \ref{lem4} and \eqref{eq} then deduce \eqref{mainthm:eq1}. 		
		Next assume that $(J,h)=(J_c,h_c).$ In this case $m_0=m_c=2-\sqrt{2}$ and a direct computation gives
		\begin{align*}
		L_1(m_c)=0,\,\,L_1'(m_c)=0,\,\,L_1''(m_c)=0,\,\,L_1'''(m_c)=6+\frac{17\sqrt{2}}{4}
		\end{align*}
		and 
		\begin{align*}
		\frac{2L_1'''(m_c)}{4!L_2(m_c)}&=\frac{1}{2}+\frac{17\sqrt{2}}{48}=-\frac{\tilde{p}^{(4)}(m_0)}{24}=\frac{\lambda_c}{24}.
		\end{align*}  
		Lemma \ref{lem4} and \eqref{eq} then yield \eqref{mainthm:eq2}. This finishes our proof.
	\end{proof}

		\section{Moment computations along the critical line $\Gamma$}
		
		This section is devoted to dealing with some moment computations for the parameters along the critical line. Set $M(\rho)=N^{-1}\sum_{i=1}^N\rho_i$ for $\rho=(\rho_1,\ldots,\rho_N)\in\Sigma.$ Let
		\begin{align*}
		S_1&=\Bigl\{\rho\in\Sigma:M(\rho)<\xi N\Bigr\},\\
		S_2&=\Bigl\{\rho\in \Sigma:M(\rho)>\xi N\Bigr\}.
		\end{align*}
		For $\ell\geq 1,$ define the probability measure $\p_\ell=\p(\,\cdot\,|S_\ell)$ on $S_\ell$. For the same reason as \eqref{eq}, one sees that 
		\begin{align}\label{eq-2}
		\p_\ell\Bigl(m(\sigma)=\frac{t}{N}\Bigr)=\frac{\p\Bigl(m(\sigma)=\frac{t}{N}\Bigr)}{\p(S_\ell)}=\frac{\p\Bigl(m(D)=\frac{t}{N}\Bigr)}{\p(\mathscr{D}_\ell)}
		=\p\Bigl(m(D)=\frac{t}{N}\Big|\mathscr{D}_\ell\Bigr),
			\end{align}
			where $\mathscr{D}_\ell$ is defined in \eqref{setd}.
	    Therefore, to prove the conditional central limit theorem for the monomer density in Theorem \ref{thm4}, it suffices to establish the central limit theorem for the magnetization under $\p_\ell$. Following the same construction as before, let $\sigma$ be sampled from $ {\p}_\ell$ and $uv$ be a uniform random variable from $E.$ Under the probability measure $\p_\ell$, let $(\sigma_u',\sigma_v')$ be the conditional distribution of $(\sigma_u,\sigma_v)$ given $(\sigma_i)_{i\neq u,v}$ and independent of $(\sigma_v,\sigma_v),$ that is,
		\begin{align*}
		{\p}_\ell(\sigma_u'=s,\sigma_v'=t|\sigma)&=\frac{ {\p}_\ell(\sigma_{uv}^{st})}{ {\p}_\ell(\sigma_{uv}^{11})+ {\p}_\ell(\sigma_{uv}^{10})+ {\p}_\ell(\sigma_{uv}^{01})+ {\p}_\ell(\sigma_{uv}^{00})}.
		\end{align*}
		This pair $(\sigma,\sigma')$ is therefore exchangeable from Proposition \ref{prop0}.
		 Set
		\begin{align*}
		S_1'&=\Bigl\{\sigma:M(\sigma)<\xi N-2\Bigr\},\\
		S_1''&=\Bigl\{\sigma:\xi N-2\leq M(\sigma)<\xi N\Bigr\}
		\end{align*}
		and
		\begin{align*}
		S_2'&=\Bigl\{\sigma:M(\sigma)>\xi N+2\Bigr\},\\
		S_2''&=\Bigl\{\sigma:\xi N+2\geq M(\sigma)>\xi N\Bigr\}.
		\end{align*}
		In the case $\sigma\in S_\ell',$ we have
		\begin{align}\label{ae:eq1}
		{\p}_\ell(\sigma_u'=s,\sigma_v'=t|\sigma)&=\frac{ {\p}(\sigma_{uv}^{st})/\p(S_\ell)}{ \bigl({\p}(\sigma_{uv}^{11})+ {\p}(\sigma_{uv}^{10})+ {\p}(\sigma_{uv}^{01})+ {\p}(\sigma_{uv}^{00})\bigr)/\p(S_\ell)}={\p}(\sigma_u'=s,\sigma_v'=t|\sigma),
		\end{align}
		where we used $\sigma_{uv}^{11},\sigma_{uv}^{10},\sigma_{uv}^{01},\sigma_{uv}^{00}\in S_\ell.$ However, if $\sigma\in S_\ell''$, this equation is no longer valid.

		\begin{lemma}\label{lem-1}
			Recall $L_k$ and $R_k$ from Proposition \ref{lem1}.
			We have that 
			\begin{align}\label{lem-1:eq1}
			\e_\ell[(M-M')^k|M]&=L_k(m(\sigma))+R_k(m(\sigma))+T_{\ell,k}(m(\sigma)),
			\end{align}
			where $ {\e}_\ell|T_{\ell,k}(m(\sigma))|\leq  K{\p}_\ell( \sigma \in S_\ell'').$
		\end{lemma}
		
		\begin{proof}
			Observe that if $\sigma\in S_\ell',$ from \eqref{ae:eq1} and then \eqref{lem1:eq1}, \eqref{lem2:eq1},
			\begin{align*}
			{\e}_\ell[(M-M')^k|\sigma]&=\frac{1}{|E|}\sum_{uv\in E} {\e}_\ell[(\sigma_u+\sigma_v-\sigma_u'-\sigma_v')^k|\sigma]\\
			&=\frac{1}{|E|}\sum_{uv\in E}\e[(\sigma_u+\sigma_v-\sigma_u'-\sigma_v')^k|\sigma]\\
			&=\e[(M-M')^k|\sigma]\\
			&=L_k(m(\sigma))+R_k(m(\sigma)).
			\end{align*}
			Thus,
			\begin{align*}
			\e_\ell[(M-M')^k|\sigma]&=1(\sigma\in S_\ell')\bigl(L_k(m(\sigma))+R_k(m(\sigma))\bigr)+ 1(\sigma\in S_\ell''){\e}_\ell[(M-M')^k|\sigma]\\
			&=L_k(m(\sigma))+R_k(m(\sigma))+T(m(\sigma)),
			\end{align*}
			where 
			\begin{align*}
			T(m(\sigma))&=1(\sigma\in S_\ell'')\bigl( {\e}_\ell[(M-M')^k|\sigma]-L_k(m(\sigma))-R_k(m(\sigma))\bigr).
			\end{align*}
			Taking conditional expectation $ {\e}_\ell[\,\cdot\,|M]$ and letting $T_{\ell,k}(m(\sigma))=\e_\ell[T(m(\sigma))|M]$ give \eqref{lem-1:eq1}. Here $\e_\ell|T_{\ell,k}(m(\sigma))|\leq K\p_\ell(\sigma\in S_\ell'')$ holds true by applying the trivial bound $|M-M'|\leq 2$ and the fact that $L_k,R_k$ are bounded.
		\end{proof}

		\section{Proof of Theorem \ref{thm4}}
		
			In this section, we suppose that $(J,h)\in \Gamma$. This assumption implies that $\tilde{p}'(m_\ell)=0$ and $\tilde{p}''(m_\ell)<0$ and therefore using the first half of the derivation of the proof of Theorem \ref{mainthm}, they can be transferred as $L_1(m_\ell)=0$ and $L_1'(m_\ell)>0$. Furthermore, the quantity $\lambda_\ell$ defined in Theorem \ref{thm4} is positive. Denote
			\begin{align*}
			W_\ell&=\frac{M-Nm_\ell}{N^{1/2}}\,\,\mbox{and}\,\,W_\ell'=\frac{M'-Nm_\ell}{N^{1/2}}.
			\end{align*} First by adapting exactly the same argument as the proof of Lemma \ref{lem3} and applying Lemma \ref{lem-1}, one obtains an analogue of Lemma \ref{lem3} for the conditional probability $\p_\ell$, whose proof will be omitted.
			 
			\begin{lemma}\label{lem-2}
		    We have
			\begin{align}\label{lem-2:eq1}
			\e_\ell[W_\ell-W_\ell'|W_\ell]&=g_\ell(W_\ell)+r_\ell(W_\ell),
			\end{align} 
			where 
			\begin{align*}
			g_\ell(W_\ell)&=\frac{L_1'(m_\ell)}{N}W_\ell
			\end{align*}
			and $r_\ell$ satisfies
			\begin{align*}
			|r_\ell(W_\ell)|&\leq \frac{K}{N^{3/2}}(W_\ell^2+1)+\frac{1}{N^{1/2}}|T_{\ell,1}(m(\sigma))|.
			\end{align*}
			In addition,
			\begin{align*}
			\e_\ell\Bigl|1-\frac{c_{0,\ell}}{2}\e_\ell[(W_\ell-W_\ell')^2|W_\ell]\Bigr|&\leq K\Bigl(\frac{1}{N^{1/2}}+\frac{1}{N}+\frac{1}{N}\p_\ell(S_\ell'')\Bigr),
			\end{align*}
			where $c_{0,\ell}={2N}/{L_2(m_\ell)}.$
			\end{lemma}
			
		   The next lemma below will play a similar role as Lemma \ref{lem5}.

	         	\begin{lemma}
	         		\label{lem-3}
	         		We have that $\e_\ell|W_\ell|^2\leq K$ for all $N\geq 1.$
	         	\end{lemma}
	         	
	         	\begin{proof}
	         		Similar to the proof of Lemma \ref{lem5}, multiplying $W_\ell$ on both sides of \eqref{lem-2:eq1} and then taking expectation give
	         		\begin{align}
	         		\begin{split}
	         		\notag
	         		\e_\ell W_\ell^2&=\frac{N}{L_1'(m_\ell)}\Bigl(\e_\ell[(W_\ell-W_\ell')W_\ell]-\e_\ell W_\ell r_\ell(W_\ell)\Bigr)
	         		\end{split}\\
	         		\begin{split}
	         		\label{proof:eq4}
	         		&\leq \frac{1}{L_1'(m_\ell)}\Bigl(N\e_\ell[(W_\ell-W_\ell')W_\ell]+\frac{K\e_\ell |W_\ell|^{3}}{N^{1/2}}+\frac{K\e_\ell|W_\ell|}{N^{1/2}}+KN^{1/2}\p_\ell (\sigma\in S_\ell'')\Bigr).
	         		\end{split}
	         		\end{align}
	         		Let us now bound each term on the right-hand side as follows. Using the exchangeability of $(W_\ell,W_\ell')$ under $\p_\ell$ and the bound $|W_\ell-W_\ell'|\leq 2/N^{1/2}$, we obtain the control of the first term
	         		\begin{align}\label{proof:eq-7}
	         		N\e_\ell (W_\ell-W_\ell')W_\ell&=\frac{N}{2}\e_\ell(W_\ell-W_\ell')^2\leq 2.
	         		\end{align}
	         		For the third term, it can be easily controlled 
	         		\begin{align}
	         		\label{proof:eq7}
	         		\frac{\e_\ell|W_\ell|}{N^{1/2}}\leq 1
	         		\end{align}
	         		 by noting $|W_\ell|\leq N^{1/2}$. As for the other two terms, note that since $\p(\sigma\in S_\ell)$ converges to $p_\ell>0$ as $N\rightarrow\infty$, this together with Lemma \ref{prop3} implies that for any $\delta>0,$ there exists some $\eta>0$ and $K>0$ such that
	         		\begin{align*}
	         		{\p}_1(|m(\sigma)-m_1|\geq \delta)\leq \frac{\p(|m(\sigma)-m_1|\geq \delta, m(\sigma)<\xi)}{\p(S_1)}\leq Ke^{-\eta N},\\
	         	    {\p}_2(|m(\sigma)-m_2|\geq \delta)\leq \frac{\p(|m(\sigma)-m_2|\geq \delta, m(\sigma)>\xi)}{\p(S_2)}\leq Ke^{-\eta N}
	         		\end{align*}
	         		for all $N\geq 1.$ Consequently, these imply that for $N$ large enough,
	         		\begin{align}
	         		\begin{split}	         		\label{proof:eq5}
	         		\p_\ell(\sigma\in S_\ell'')&\leq \p_\ell(|m(\sigma)-\xi|\leq 2/N)\leq Ke^{-\eta N}.
	         		\end{split}
	         		\end{align}
	         	    Finally, proceeding in the same way as \eqref{lem3:proof:eq1}, we see that
	         		\begin{align}\label{proof:eq6}
	         		\frac{\e_\ell |W_\ell|^{3}}{N^{1/2}}&\leq \delta \e_\ell|W_\ell|^{2}+KNe^{-\eta N}.
	         		\end{align}
	         		Plugging  \eqref{proof:eq-7}, \eqref{proof:eq7}, \eqref{proof:eq5} and \eqref{proof:eq6} into \eqref{proof:eq4} gives
	         		\begin{align*}
	         		\Bigl(1-\frac{K}{L_1'(m_\ell)}\delta\Bigr)\e_\ell|W_\ell|^2&\leq \frac{K'}{L_1'(m_\ell)}\Bigl(1+Ne^{-\eta N}+N^{1/2}e^{-\eta N}\Bigr),
	         		\end{align*}
	         		for some $K'>0.$ Note that $K$ is independent of $\delta.$ We can choose $\delta$ small enough so that the announced statement holds. 
	         	\end{proof}
	
		\begin{proof}[Proof of Theorem \ref{thm4}]
        Recall $c_{0,\ell}$ and $g_{\ell}$ from \eqref{lem-2}. Define
        $$
        p_\ell(t)=c_{1,\ell}e^{-c_{0,\ell}\int_0^tg_\ell(s)ds}=c_{1,\ell}e^{-d_\ell t^{2}},
        $$			
        where $c_{1,\ell}$ is the normalizing constant such that $p_\ell$ forms a probability density on $\mathbb{R}$ and $d_\ell=L_1'(m_\ell)/L_2(m_\ell)>0.$ Observe that 
        \begin{align*}
        &c_{0,\ell}|g_\ell'(x)|\Bigl(|x|+\frac{3}{c_{1,\ell}}\Bigr)\min\Bigl(\frac{1}{c_{1,\ell}},\frac{1}{|c_{0,\ell}g_{\ell}(x)|}\Bigr)\\
        &=2d_\ell|x|\Bigl(|x|+\frac{3}{c_{1,\ell}}\Bigr)\min\Bigl(\frac{1}{c_{1,\ell}},\frac{1}{d_\ell x^2}\Bigr)
        \end{align*}
        has a limit at infinity and is clearly bounded for arbitrary small $x.$ This function has a uniform upper bound over $\mathbb{R}$ that is denoted by $c_{2,\ell}.$ Let $\delta_\ell=2/N^{1/2}$ and $\Delta_\ell=W_\ell-W_\ell'.$ Applying these $c_{0,\ell},c_{1,\ell},c_{2,\ell},g_\ell,r_\ell,p_\ell$ and $(W_\ell,W_\ell')$ under $\p_\ell,$ the inequality \eqref{thm1:ass} leads to
        \begin{align*}
        	&\sup_z|\p_\ell(W_\ell\leq z)-\p(X_\ell\leq z)|\\
        	&\leq 3\e_\ell\Bigl|1-\frac{c_{0,\ell}}{2}\e_\ell[\Delta_\ell^2|W_\ell]\Bigr|+\frac{2c_{0,\ell}}{c_{1,\ell}}\e_\ell|r_\ell(W_\ell)|\\
        	&+c_{1,\ell}\max(1,c_{2,\ell})\e_{\ell}|\delta_\ell| +|\delta_\ell|^3c_{0,\ell}\Bigl\{\Bigl(2+\frac{c_{2,\ell}}{2}\e_\ell |c_{0,\ell}g_\ell(W_\ell)|\Bigr)+\frac{c_{1,\ell}c_{2,\ell}}{2}\Bigr\}\\
        	&\leq K\Bigl(\frac{1}{N^{1/2}}\e_\ell|W_\ell^2|+\frac{1}{N}+\frac{1}{N^{1/2}}+\frac{1}{N^{3/2}}+N^{1/2}\p_\ell(\sigma\in S_\ell'')+\frac{1}{N}\p_\ell(\sigma\in S_\ell'')\Bigr).
        \end{align*}
       Here the second inequality used Lemmas \ref{lem-1} and \ref{lem-2} and $|\Delta_\ell|\leq \delta_\ell.$ Finally, using Lemma \ref{lem-3}, the inequality \eqref{proof:eq5}  and the identity \eqref{eq-2} finishes our proof.
		\end{proof}

\end{document}